\documentclass{article}
\usepackage[T1]{fontenc}
\usepackage[utf8]{inputenc}
\usepackage{amsmath,amsfonts,amssymb,bbold,bbm,textcase,amsthm}

\def\({\left(}
\def\){\right)}

\def\ep{\varepsilon}
\def\bar{\overline}

\def\Sign{\mathrm{Sign}}

\def\mc#1{\mathcal{#1}}
\def\<{\left<} \def\>{\right>}
\def\RR{\mathbbm R} \def\R{\mathcal R}
 
\def\F{\mathcal F} 
\renewcommand{\restriction}{|} 

\usepackage[frame,arrow,curve]{xy}
\newenvironment{grafo}[1]{\begin{xy}0;<#1,0cm>: }{\end{xy}}
\newcommand{\Nodo}[5][5mm]{\POS *=<#1>[o]{#3}="#2"+#4(1.5)*{#5}}
\newcommand{\NodoO}[5][5mm]{\POS *=<#1>[o]{#3}*\frm{o}="#2"+#4(1.5)*{#5}}
\newcommand{\arcoD}[3][]{\ar@{-}"#2";"#3"#1}
\newcommand{\arcoDot}[3][]{\ar@{.}"#2";"#3"#1}
\newcommand{\arcoDotU}[3][]{\ar@(ur,ul)@{-}"#2";"#3"#1}
\newcommand{\loopU}[2][]{\ar@(ur,ul)@{-}"#2";"#2"#1}
\newcommand{\loopD}[2][]{\ar@(dr,dl)@{-}"#2";"#2"#1}

\newtheorem{theorem}{Theorem}[section]
\newtheorem{lemma}[theorem]{Lemma}
\newtheorem{corollary}[theorem]{Corollary}

\newtheorem{definition}[theorem]{Definition}

\def\E{\mathcal{E}}
\def\V{\mathcal{V}}
\def\SS{\mathcal{S}}

\def\L{\mathcal{L}}
\def\F{\mathcal{F}}

\newcommand{\norm}[1]{\left\|#1\right\|}

\def\L{\mathit{\Delta}}

\usepackage[usenames,dvipsnames]{color}

\usepackage{titlesec}
\titleformat*{\section}{\large\bfseries}
\titleformat*{\subsection}{\itshape}
\titleformat*{\subsubsection}{\footnotesize\bfseries}

\begin{document}
\title{A nodal domain theorem and a higher-order Cheeger inequality for the graph \lowercase{$p$}-Laplacian}

\author{Francesco Tudisco\footnote{tudisco@cs.uni-saarland.de, corresponding author}, Matthias Hein\footnote{hein@math.uni-sb.de}\\[5px]
Department of Mathematics and Computer Science, Saarland University,\\ Campus E1 1, 66123 Saarbr\"ucken, Germany}

\date{}

\maketitle

\begin{abstract}
We consider the nonlinear graph $p$-Laplacian and its set of eigenvalues and associated eigenfunctions of this operator defined by a variational principle. We prove a nodal domain theorem  for the graph $p$-Laplacian for any  $p\geq 1$. While for $p>1$ the bounds on the number of weak and strong nodal domains are the same as for the linear graph Laplacian ($p=2$), the behavior changes for $p=1$. We show
that the bounds are tight for $p\geq 1$ as the bounds are attained by the eigenfunctions of the graph $p$-Laplacian on two graphs. Finally, using the properties of the nodal domains,
we prove a higher-order Cheeger inequality for the graph $p$-Laplacian for $p>1$. If the eigenfunction associated to the $k$-th variational eigenvalue of the graph $p$-Laplacian has exactly $k$ strong nodal domains, then the higher order Cheeger inequality becomes tight as $p\rightarrow 1$.
\end{abstract}

 		
\section{Introduction}
In this paper we are concerned with the $p$-Laplacian operator $\L_p$ on finite, undirected, weighted graphs $G=(\V,\E)$. For simplicity we assume throughout the paper that the graph $G$ is connected. A function $f$ on $\V$ is an eigenfunction of $\L_p$ corresponding to the eigenvalue $\lambda$ if it solves the following eigenvalue problem
\begin{equation}\label{eq:eigenequation}
 (\L_p f)(u)= \lambda \, \mu(u)|f(u)|^{p-2}f(u), \quad \forall u \in \V
\end{equation}
The eigenvalue problem for the linear graph Laplacian is obtained for $p=2$. The spectrum of $\L_2$ has been studied extensively in past decades. In particular every eigenvalue of $\L_2$ admits a variational characterization in terms of a Rayleigh quotient and several relations between the eigenvalues of $\L_2$ and a number of graph invariants have been established \cite{chung, chung-advmath,expanders}. However, for $p\neq 2$, the spectral properties are less well understood. For the general case it is known that the smallest eigenvalue $\lambda_1$ of $\L_p$ is zero, it is simple and  any corresponding eigenfunction is constant. It is also known that the smallest nonzero eigenvalue $\lambda_2$ admits a variational characterization \cite{Amghibech2003, Anane-Tsouli,Thomas2009}.  The Lusternik-Schnirelman theory allows to generalize the variational characterization of the linear case and to define a sequence of variational eigenvalues of $\L_p$ for $p\neq 2$. 

In this paper we investigate the nodal properties of the eigenfunctions of $\L_p$.
A strong nodal domain of $f$ is a maximal connected component of $\{u: f(u)\neq 0\}$. Weak nodal domains, instead, can overlap and are defined as the maximal connected components of either $\{u : f(u)\geq 0\}$ or $\{u: f(u)\leq 0\}$.  The famous Courant nodal domain theorem provides upper bounds on the number of nodal domains of the eigenfunctions of the continuous Laplacian in $\RR^d$. Several authors worked afterwards on a discrete version of the nodal domain theorem for the case of the linear graph Laplacian, see e.g. 
\cite{nodal-domain-theorem, Verdiere1993,duval-nodal-domains,fiedler-vector}, or the adjacency or modularity matrix \cite{Fasino2014,powers-graph-eigenvector}.
The main contribution of the present work is a unifying generalized version  of the Courant nodal domain theorem for the graph $p$-Laplacian, for any $p\geq 1$. We show that for $p>1$ the bounds on the number of weak and strong nodal domains are the same as for the linear graph Laplacian whereas the upper bound of the weak nodal domains changes for $p=1$. 

As there are strong relations between the continuous and discrete theory, it is worthwhile to quickly review the eigenproblem of the continuous $p$-Laplacian.
If $\Omega$ is a bounded domain in $\RR^d$, with smooth boundary $\partial \Omega$,  the continuous analogous of \eqref{eq:eigenequation} is 
\begin{equation}\label{eq:continuous_eigenequation}
 -\mathrm{div}(|\nabla f|^{p-2}\nabla f)=\lambda |f|^{p-2}f,  \qquad \mathrm{in} \, \, \Omega
\end{equation}
where homogeneous conditions are assumed on $\partial \Omega$. The  eigenvalue problem \eqref{eq:continuous_eigenequation} has been studied extensively. When $d =1$, for instance, the spectrum of $\L_p$ is completely described \cite{Drabek1992,Elbert1979}. 
Courant's nodal theorem has been then extended to the eigenfunctions of the continuous $p$-Laplacian \eqref{eq:continuous_eigenequation} for $p>1$ \cite{Anane-Tsouli, Drabek2002}, where connected components are replaced by connected open subsets. However, note that the difference between strong and weak nodal domains is not considered in the
continuous case. Moreover, \cite{Drabek2002} require as an assumption what they call the \emph{unique continuation property} to prove the direct generalization of the Courant nodal domain theorem.

As a second main contribution, we provide a higher-order Cheeger inequality relating the $k$-th variational eigenvalue of the graph $p$-Laplacian with the $k$-th isoperimetric
or Cheeger constant of the graph. The Cheeger constant $h(G)$ is one of the most important topological invariants of a graph $G$. 
The result for the case $k=2$ was originally proven Cheeger \cite{cheeger} for compact Riemannian manifolds and the associated Laplace-Beltrami operator.

	Several authors extended afterwards Cheeger result to the discrete case, see e.g. \cite{Alon1986,Alon1985,Dodziuk1984,Fasino2014,Lawler1988,Mohar1989,Sinclair1989}. 
	Cheeger's inequality plays an important role in the theory of expander graphs, mixing time of Markov chains, graph coloring but has also applications in computer science such as image segmentation and web search, see e.g.\ \cite{Alon1985,Chung2010,Kwok2013,Lee2012} and the references therein. An extension of Cheeger's inequality to the nonlinear graph $p$-Laplacian has been shown in \cite{Amghibech2003,Thomas2009,Hein2010} for $p\geq 1$, where it has been shown that the Cheeger constant $h(G)$ of the graph is the limit of $\lambda_2$ of $\L_p$ as $p\to 1$, with equality for $p=1$. A number of Cheeger type inequalities have been shown for  the continuous $p$-Laplacian in \eqref{eq:continuous_eigenequation} as well \cite{Caroccia2015,Kawohl2008,Keller2015}.
	
More recently, a set of high-order isoperimetric constants $h_k(G)$, $k = 1,2,3,\dots$ alternatively called multi-way Cheeger constants,  has been introduced by Miclo \cite{Daneshgar2010,Miclo2012} in the discrete setting.	We provide a Cheeger-type inequality relating the $k$-th variational eigenvalue of $\L_p$ with the $k$-way isoperimetric constant, where the number of strong nodal domains of the eigenfunctions of $\L_p$ plays a crucial role.

The paper is organized as follows. In Section \ref{sec:preliminaries} we fix the notation and provide a number of first results. Section \ref{sec:nodal-theorem} contains the statement of the nodal domain theorem for $\L_p$. In Sections \ref{sec:Pn} and \ref{sec:L1} we discuss the non-smooth case $p=1$ and prove the tightness of our results   for $p\geq 1$, by studying the eigenfunctions of the graph $p$-Laplacian for two graphs where the upper bounds on the nodal domain counts are attained. Finally, in Section \ref{sec:cheeger} we discuss a higher-order Cheeger inequality for the graph $p$-Laplacian.

\section{Preliminaries}\label{sec:preliminaries}
Let $G = (\V, \E)$ be a finite, connected and undirected graph where $\V=(V,\mu)$ and $\E = (E,w)$ are the vertex and edge sets, equipped with positive measures $\mu$ and $w$, respectively.
The number of nodes $|V|$ of $G$ is denoted by $n$. We assume that the vertex weights are strictly positive, $\mu(u)>0$ for all $u \in V$.
 An element of $E$ is denoted by $uv$, where $u,v$ are the nodes connected by $uv$. We extend the measure function $w$ to the whole product $V\times V$ by letting $w(uv)=0$ if $uv\notin E$.  We write $u \sim v$ to indicate that there exists an edge $uv \in E$ between those two nodes. Similarly, for two sets of nodes $A,B \subseteq V$, we write $A\approx B$ if there exists an edge connecting $A$ and $B$. Finally, we define $\ell^p(\V)$ to be the space of real valued functions on $V$ endowed with the norm $\|f\|_{\ell^p(\V)}=(\sum_{u}\mu(u)|f(u)|^p)^{1/p}$.  

Let $\Phi_p:\RR \rightarrow \RR$ defined as $\Phi_p(x)=|x|^{p-2}x$, then the graph $p$-Laplacian $\L_p:\RR^{\V} \rightarrow \RR^{\V}$ is defined for $p>1$ as
$$(\L_p f)(u) = \sum_{v \in V}w(uv)\Phi_p(f(u)-f(v)), \qquad \forall u \in V.$$
The case $p=1$ will be treated in Section \ref{sec:L1}, even though parts of the following discussion apply already for the case $p=1$.
A function $f:V\to \RR$  is an eigenfunction of $\L_p$ associated with the eigenvalue $\lambda$ if the following identity holds 
$$(\L_p f)(u) =\lambda\, \mu(u) \Phi_p(f(u)), \qquad \forall u \in V.$$
We then consider the even functional $\R_p: \RR^{\V} \rightarrow \RR$ defined as
$$\R_p(f) = \frac 1 2 \frac{\sum_{uv\in E}w(uv)|f(u)-f(v)|^p}{\sum_{u\in V} \mu(u)|f(u)|^p}$$
and the symmetric manifold $\SS_p = \{g\in \ell^p(\V) : \|g\|_{\ell^p(\V)}=1\}$. It is easy to see that the eigenvalues and eigenfunctions of $\L_p$ correspond to the critical values and critical points of $\R_p$. Moreover, as $\R_p$ is scale invariant, $\R_p(\alpha f)=\R_p(f)$ for any nonzero $\alpha \in \RR$, $f$ is a critical point of $\R_p$ if and only if $f/\|f\|_{\ell^p(\V)}$ is a critical point of the restriction $\R_p|_{\SS_p}$ of $\R_p$ onto $\SS_p$, and they correspond to the same critical value. 

The Lusternik-Schnirelman theory allows several ways to characterize a sequence of variational eigenvalues of $\L_p$. A standard approach, relaying on the symmetry of $\R_p$ and $\SS_p$, is based on the notion of Krasnoselskii genus. 
\begin{definition}
	Let $A$ be a closed, symmetric subset of $\SS_p$. We define the Krasnoselskii genus of $A$ as
	$$\gamma(A)=\begin{cases}
	0  & \text{if }A = \emptyset\\
	& \!\!\!\!\!\!\!\!\!\!\!\! \inf\{m \mid \exists h:A \to \RR^m\setminus\{0\},\, \,  continuous,\, \, h(-u)=-h(u)\} \\
	\infty & \text{if } \{...\}=\emptyset, \text{ in particular if }0\in A
	\end{cases}\,\,\, .$$
\end{definition}
Consider the family $\F_k(\SS_p)=\{A \subseteq \SS_p: A=-A, \text{closed}, \gamma(A)\geq k\}$. As  $\SS_p$ is compact, it is easy to verify that $\R_p|_{\SS_p}$ satisfies the Palais-Smale condition. Then  
\begin{equation}\label{eq:lambdas}
		\lambda_k^{(p)} = \min_{A \in \mc F_k(\SS_p)}\max_{f \in A}\, \R_p(f)
	\end{equation}
defines a sequence of $n$ critical values of the Rayleigh quotient $\R_p$, for $ k=1,\dots, n$. Moreover, it is known that  the connectedness of $G$ implies $0=\lambda_1^{(p)}<\lambda_2^{(p)}\leq \dots \leq \lambda_n^{(p)}$. Note that in the definition \eqref{eq:lambdas} we can freely use $\R_p$ or its restriction to $\SS_p$, as the critical values do not change. 
From now on we shall call the numbers  $\lambda_k^{(p)}$, for $k=1, \dots, n$, the \textit{variational eigenvalues} of $\L_p$. For ease of notation we will often drop the superscript, writing $\lambda_k$ in place of $\lambda_k^{(p)}$, when the reference to a given $p$ is clear from the context.

Until now we have concentrated on the variational eigenvalues but the nodal domain theorem requires to consider eigenfunctions. It turns out that one can associate at least one
eigenfunction of the graph $p$-Laplacian to each variational eigenvalue $\lambda_k^{(p)}$.
For a  function $F:\mc \SS_p\to \RR$ we write  $F^{c}=\{x \in \mc \SS_p : F(x)\leq c\}$ and ${K_\lambda}(F)= \{x \in \mc \SS_p : F(x)=\lambda, \nabla F(x)=0\}$. 
The following theorem is  part of a more general result known as \textit{deformation theorem}. See for instance \cite[Thm. 3.11]{Struwe2013} or \cite[Thm. 4.1.19]{Papageorgiou2009}
\begin{theorem}\label{thm:deformation}
	Let $F:\mc S_p \to \RR$ be an even function satisfying the Palais-Smale condition. Then for any $\ep_0>0$, $\lambda \in \RR$ and any neighbourhood $N$ of ${K_\lambda}(F)$, there exists $\ep \in (0,\ep_0)$ and an odd homeomorphism $\theta:\mc \SS_p\to \SS_p$ such that 
	$$\theta(F^{\lambda+\ep}\setminus N)\subseteq F^{\lambda-\ep}\, .$$
\end{theorem}
The deformation theorem allows to show that to each variational eigenvalue belongs at least one corresponding eigenfunction.
\begin{lemma}\label{thm:existence}
	For $k\geq 1$ let $A^* \in \mc F_k(\SS_p)$ be a minimizing set, that is 
	$$\lambda_k = \min_{A \in \mc F_k(\SS_p)}\max_{f \in A}\, \R_p(f)=\max_{f \in A^*}\, \R_p(f)\, .$$
	Then $A^*$ contains at least one critical point of $\R_p$, relative to $\lambda_k$.
\end{lemma}
\begin{proof}
	Consider the restriction $\R_p|_{\SS_p}$. Recall that if $F$ is a function that satisfies the Palais-Smale condition and $\lambda$ is a critical value of $F$, then $K_\lambda(F)$ is compact (see f.i.\ \cite[pp.\ 78-80]{Struwe2013}).	
	Suppose by contradiction that $A^*\cap K_{\lambda_k}(\R_p|_{\SS_p})=\emptyset$. We already discussed that  $ \R_p|_{\SS_p}$ satisfies the Palais-Smale condition, therefore $A^*$  and $K_{\lambda_k}(\R_p|_{\SS_p})$ are compact. Hence there exists a neighborhood $N$ of $K_{\lambda_k}(\R_p|_{\SS_p})$ such that $A^*\cap N=\emptyset$. Therefore $A^* = A^*\setminus N$. Since $\max_{f \in A^*}\R_p|_{\SS_p}(f) = \lambda_k$, for any $\ep>0$ we have $A^*\subseteq \R_p|_{\SS_p}^{\lambda_k+\ep}$. By the deformation theorem, there exists an odd homeomorphism $\theta:{\SS_p}\to {\SS_p}$ such that
	$$\theta(A^*) = \theta(A^*\setminus N)\subseteq \theta\Big(\R_p|_{\SS_p}^{\lambda_k+\ep}\setminus N\Big)\subseteq  \R_p|_{\SS_p}^{\lambda_k-\ep}$$
	As  $\theta$ is an odd homeomorphism we have that $A^* \in \mc F_k({\SS_p})$ implies $\theta(A^*) \in  \mc F_k({\SS_p})$. Then
	$$\lambda_k\leq \min_{A \in  \mc F_k({\SS_p})}\max_{f \in A}\, \R_p(f)\leq \max_{f \in \theta(A^*)}\R_p(f)\leq \max_{f \in  \R_p|_{\SS_p}^{\lambda_k-\ep}}\R_p(f)\leq \lambda_k-\ep$$
	and we have reached a contradiction. Therefore $A^*\cap K_{\lambda_k}( \R_p|_{\SS_p})$ cannot be empty and the lemma is proven.
\end{proof}
We would like to note that the integer valued genus $\gamma$ is a classical homeomorphism invariant generalization of the concept of dimension. Indeed if $A$ is any symmetric neighborhood  of the origin in $\RR^k$, then $\gamma(A)=k$, and, vice-versa, if $A$ is any subset such that $\gamma(A)=k$, then $A$ contains at least $k$ mutually orthogonal functions. It follows that, when $p=2$,  the sequence \eqref{eq:lambdas} boils down to the Courant-Fischer minimax principle  $\lambda_k^{(2)}=\min_{\dim(A)=k}\max_{f \in A}\R_2(f)$. We refer to  \cite[Ch.\ 4]{Papageorgiou2009} or \cite[Ch.\ 2]{Struwe2013} for an overview.

\section{Nodal domain theorem for the graph $p$-Laplacian}\label{sec:nodal-theorem}

Consider the eigenvalue problem \eqref{eq:continuous_eigenequation} and a continuous function $f$ on $\Omega$. A nodal domain for $f$ is a maximal connected open subset of $\{u:f(u)\neq 0\}$.
When $p=2$, Courant's nodal domain theorem states that any eigenfunction for \eqref{eq:continuous_eigenequation} associated to the eigenvalue $\lambda_k$ has at most $k$ nodal domains.

For graphs, nodal domains induced by a function $f:V\to \RR$ are commonly defined as follows: 
\begin{definition}   \label{def:strong_nd}
	Let $f:V\to \RR$. 
	A subset $A \subseteq V$ is a \emph{strong nodal domain} of $G$ induced by $f$ if the subgraph $G(A)$ induced on $G$ by $A$ 
	is a maximal connected component of either $\{u: f(u) > 0\}$ or  $\{u: f(u) < 0\}$.
\end{definition}

\begin{definition}   \label{def:weak_nd}
	Let $f:V\to \RR$. 
	A subset $A \subseteq V$ is a \emph{weak  nodal domain} of $G$ induced by $f$ if the subgraph $G(A)$ induced on $G$ by $A$ 
	is a maximal connected component of either 
	$\{u: f(u) \geq 0\}$  or $\{u: f(u) \leq 0\}$. 
\end{definition}

For any connected graph $G$ and any $p \geq 1$, $\lambda_1^{(p)} = 0$ is simple and any associated eigenfunction is constant.
Thus the strong and weak nodal domain for the eigenfunctions of $\lambda_1^{(p)}$ is $V$ itself. 

Fiedler noted in \cite[Cor.\ 3.6]{fiedler-vector}
that the number of weak nodal domains induced by 
any eigenfunction associated to $\lambda_2^{(2)}$
 is exactly two. 
Several authors derived analogous results to the Courant nodal theorem for the higher-order eigenfunctions of  $\L_2$
\cite{nodal-domain-theorem, duval-nodal-domains, powers-graph-eigenvector}.  
The following nodal domain theorem for the graph Laplacian $\L_2$  summarizes their work:
\begin{theorem}\label{thm:laplacian-domains}
	Let $G$ be connected and $0=\lambda_1<\lambda_2\leq \dots \leq \lambda_n$ 
	be the eigenvalues of $\L_2$. Any eigenfunction of $\lambda_k$ induces at most $k$ weak nodal domains and at most $k+r-1$ strong nodal domains, where $r$ is the multiplicity of $\lambda_k$.
\end{theorem}
The authors of \cite{nodal-domain-theorem}, in particular, provide examples which show that the bounds for the weak and strong nodal domains are tight.
The following theorems
show that the results carry over to the $p$-Laplacian.
\begin{theorem}\label{thm:p_laplacian-strong-domains}
 Suppose that $G$ is connected and $p\geq 1$ and denote by $0=\lambda_1<\lambda_2\leq \dots \leq \lambda_n$ 
 the variational eigenvalues of $\L_p$. Let $\lambda$ be an eigenvalue of $\L_p$ such that $\lambda < \lambda_k$. Any eigenfunction associated to $\lambda$ induces at most $k-1$ strong nodal domains.
\end{theorem}
We get as a consequence that, if the variational eigenvalue $\lambda_k$ has multiplicity $r$, that is
	\[ \lambda_{k-1} < \lambda_k = \lambda_{k+1} = \ldots = \lambda_{k+r-1} < \lambda_{k+r}\, ,\]
	then Theorem \ref{thm:p_laplacian-strong-domains} shows that  any eigenfunction of $\lambda_k$ induces at most $k+r-1$ strong nodal domains. 
\begin{theorem}\label{thm:p_laplacian-weak-domains}
	Suppose that $G$ is connected and $p>1$ and denote by $0=\lambda_1<\lambda_2\leq \dots \leq \lambda_n$ 
 the variational eigenvalues of $\L_p$. Any eigenfunction of $\lambda_k$ induces at most $k$ weak nodal domains. 
\end{theorem}
Let us stress that Theorem \ref{thm:p_laplacian-strong-domains} holds for any $p\geq 1$, whereas Theorem \ref{thm:p_laplacian-weak-domains} does not hold in general when $p=1$. We will discuss the case $p=1$ in detail in Section \ref{sec:L1}. 
As a direct consequence we get the following corollary.
\begin{corollary}\label{cor:second-eigenvector}
	Suppose that $G$ is connected and let $p>1$. Any eigenfunction corresponding to the second variational eigenvalue of $\L_p$ has exactly $2$ weak  nodal domains.
\end{corollary} 
\begin{proof}
With the definition of $\L_p$ we have $\sum_u (\L_p f)(u)=0$ for any function $f$. 
This implies in particular that for any eigenfunction $f$ of $\L_p$ with eigenvalue not equal to zero 
it holds $\sum_{u} \mu(u)\Phi_p( f(u))=0$ which implies that $f$ attains both positive and negative values. 
As the graph is connected, it holds $\lambda_2>0$ and thus  any associated eigenfunction has at least two weak nodal domains
On the other hand Theorem \ref{thm:p_laplacian-weak-domains} shows that the number of  weak nodal domains induced by $f$ is at most $2$, and thus  it is exactly $2$. 
\end{proof}
The proof of  Theorems \ref{thm:p_laplacian-strong-domains}  and \ref{thm:p_laplacian-weak-domains} relies on a number of properties which are of independent interest. Therefore we devote the subsequent discussion to those properties and postpone the proof to the end of the section.



We need, first, the following technical lemma  
\begin{lemma}\label{lem:ax-by}
	Let $p\geq 1$, $a,b,x,y \in \RR$ and $xy \leq 0$. Then
\[ |ax-by|^p - (|a|^p|x|+|b|^p|y|)|x-y|^{p-1}\, \leq 0,\]
where equality holds for $p=1$ if and only if $xy=0$ or $ab\geq 0$ and for $p>1$ if and only if $xy=0$ or $a=b$.
\end{lemma}
\begin{proof}
We note that with $xy\leq 0$ it holds $|x-y|=|x|+|y|$. It is easy to see that equality holds for $xy=0$. Thus
we assume $xy<0$ in the following and get,
\begin{align*}
	     &|ax-by|^p - (|a|^p|x|+|b|^p|y|)(|x|+|y|)^{p-1}\\
	\leq & \big(|a||x|+|b||y|\big)^p - (|a|^p|x|+|b|^p|y|)(|x|+|y|)^{p-1}\\
   = & (|x|+|y|)^p \Big[ \Big( \frac{|x|}{|x|+|y|} |a| + \frac{|y|}{|x|+|y|} |b|\Big)^p - \frac{|x|}{|x|+|y|} |a|^p - \frac{|y|}{|x|+|y|}|b|^p \Big]\leq 0,
\end{align*}
where in the last inequality we have used the fact that $f(\lambda)=\lambda^p$ is strictly convex on $\RR_+$ for $p>1$ and convex for $p=1$. Finally, under the condition $xy< 0$ we have equality in the first inequality if and only if $ab\geq 0$. 
For the second inequality we note that it is an equality for $p=1$, whereas, under the condition $xy<0$,  equality holds for $p>1$  only if $|a|=|b|$, due to the strict convexity of $f(\lambda)=\lambda^p$. Combining the conditions yields the result.
\end{proof}

Given a function $f:V\to \RR$ and any $A \subseteq V$ we write $f|_A$ to denote the function $f|_A(u)=f(u)$ if $u \in A$ and $f|_A(u)=0$ otherwise. The strong and weak nodal spaces of $f$ are defined as the linear span of $f|_{A_1}$, $\dots$, $f_{A_m}$, being $A_i$  the strong or weak nodal domains of $f$, respectively. 
A related version of this result has been proven in \cite{duval-nodal-domains} for the linear case ($p=2$). Even though the proof there relied on the the linearity of the operator,
it turns out that this requirement is not necessary for the nonlinear generalization.

\begin{lemma}\label{le:NDS}
	Let $p\geq 1$ and let $f:V\to \RR$ be any eigenfunction of $\L_p$ corresponding to the eigenvalue $\lambda$.  Let $F$ be either the strong or weak nodal space of $f$. Then
	for any $g \in F$ it holds $\R_p(g)\leq \lambda$. In the case $p=1$ the inequality holds with equality for any $g \in F$ with $g\neq 0$.
\end{lemma}
\begin{proof}
We prove the lemma for the strong nodal domains  $A_1,\ldots,A_m$. 
We discuss at the end of the proof how it can be transferred to the weak nodal domains.  Note that the strong nodal domains are by construction pairwise disjoint. We denote by
$Z=V \backslash \cup_{i=1}^m A_i$ the set $Z=\{u :  f(u)=0\}$.  Let $g = \sum_i \alpha_i f|_{A_i}$ be a function in the strong nodal space $F$. The statement is trivially true if $g\equiv 0$, therefore we can assume $\sum_i|\alpha_i|>0$. We have
%
%
%
%
\begin{align}\label{eq:uno}
	\|g\|_{\ell^p(\V)}^p = \sum_{i=1}^m \sum_{u \in A_i}\mu(u)\big|\alpha_i\,  f|_{A_i}(u)\big|^p=\sum_{i=1}^m|\alpha_i|^p\, \big\|f|_{A_i}\big\|_{\ell^p(\V)}^p\, .
%
%
	\end{align}
By splitting the summation over $V$ into the sum over $Z, A_1, \dots, A_m$, we get
\begin{align}\label{eq:due}
	\frac{1}{2}\sum_{u, v\in V}w(uv)|g(u)-g(v)|^p &=
	\frac{1}{2}\sum_{i=1}^m |\alpha_{i}|^p \, \sum_{u,v \in A_i}  w(uv)\big|f|_{A_i}(u)-f|_{A_i}(v) \big|^p \nonumber \\
	 &+ \frac{1}{2}\sum_{j \neq i} \sum_{u \in A_j} \sum_{v \in A_i} w(uv) \big|\alpha_j f|_{A_j}(u) - \alpha_i f|_{A_i}(v) \big|^p \nonumber \\
	 &+ \sum_{i=1}^m \sum_{u \in A_i} |\alpha_i f|_{A_i}(u)|^p \sum_{v \in Z} w(uv).
\end{align}
	Let $A$ be any strong nodal domain. As $f$ is an eigenfunction of $\L_p$ corresponding to the eigenvalue $\lambda$, for any $u \in A$ we have the chain of equalities $\lambda\mu(u)|f|_A(u)|^p = \lambda \mu(u )f|_A(u)\Phi_p( f(u))=	f|_{A}(u)(\L_p f)(u)$. Therefore 
	\begin{align*}
	\lambda\left\|f|_{A}\right\|^p_{\ell^p(\V)} &= \sum_{u \in V} f|_{A}(u)(\L_p f)(u)\\
	&= \frac{1}{2}\hspace{-1mm}\sum_{u,v\in V}w(uv)(f|_{A}(u)-f|_{A}(v))\Phi_p(f(u)-f(v))
\end{align*}
Let $A,B \subset V$ be two distinct strong nodal domains. If $uv \in E$,  $u \in A$ and $v \in B$, then $f(u) f(v) < 0$, as the strong nodal domains are maximal connected components. This implies that, for such $u$ and $v$,  $\mathrm{sign}(f(u)-f(v))=
\mathrm{sign}(f|_{A}(u))=-\mathrm{sign}(f|_{B}(v))$.
Thus
\begin{align*}
\lambda \norm{f|_A}_{\ell^p(\V)}^p =& 
	 \frac{1}{2}\sum_{u,v \in A}w(uv)|f|_A(u) -f|_A(v) |^p +
	 \sum_{u \in A} |f|_A(u)|^p \sum_{v \in Z} w(uv)\\
	 +&\frac{1}{2}\sum_{B: B\neq A}\sum_{u \in A}\sum_{v \in B} \big(w(uv) |f|_A(u)|+w(vu)|f|_A(v)|\big)\,  |f|_A(u)-f|_B(v)|^{p-1}
	\end{align*}
	where the summation over $B$ runs over all the nodal domains different from $A$. Combining  the preceding formula with \eqref{eq:uno} and \eqref{eq:due} yields 
	\begin{equation}\label{eq:tre}
	       \frac{1}{2}\sum_{u,v\in V}w(uv)|g(u)-g(v)|^p- \lambda	\|g\|_{\ell^p(\V)}^p = \frac{1}{2}\sum_{i \neq j} \sum_{u \in A_i}\sum_{v \in A_j} w(uv) F_{ij}(u,v)
	\end{equation}
	where 
	$$F_{ij}(u,\!v)\!=\!
	\big|\alpha_i f|_{A_i}\!(u) - \alpha_j f|_{A_j}\!(v)\big|^p \!- \Big(\!|\alpha_i|^p\big|f|_{A_i}\!(u)\big|+ |\alpha_j|^p\big|f|_{A_j}\!(v)\big|\!\Big) \big|f|_{A_i}\!(u) - f|_{A_j}\!(v)\big|^{p-1}$$
  By Lemma \ref{lem:ax-by} each of the quantities $F_{ij}(u,v)$ is nonpositive. Since for distinct domains $A$ and $B$, $w(uv)>0$ holds  if and only if  $f|_A(u) f|_B(v) < 0$, we deduce that the quantity in \eqref{eq:tre} is nonpositive as well. As $g$ is not identically zero we conclude that $\R_p(g)\leq \lambda$. 
  Also note that, by Lemma \ref{lem:ax-by}, we have the equality $\R_p(g)=\lambda$ when $p=1$.

The proof can be transferred to the weak nodal domains $A_1,\ldots,A_m$ by considering instead the sets
$B_i=A_i \cap \{u :f(u)\neq 0\}$, $i=1,\ldots,m$ and noting that
\[ \sum_{k=1}^m \alpha_k \,    f\restriction_{A_k} = \sum_{k=1}^m \alpha_k \,   f\restriction_{B_k} .\]
As for the strong nodal domains, the sets $B_1,\ldots,B_m$ are pairwise disjoint and together with $Z=V \backslash \cup_{i=1}^m B_i
 =\{u:f(u)=0\}$, form a partition of $V$. Replacing $A_1,\ldots,A_m$ with $B_1,\ldots,B_m$  in the argument above  all the steps remain true.
\end{proof}


 We are now ready to prove the nodal domain theorem for the graph $p$-Laplacian. The proof is given here assuming $p>1$. The case $p=1$ is discussed in Section \ref{sec:L1}. 
 
 	\begin{proof}[Proof of Theorem \ref{thm:p_laplacian-strong-domains}]
 	Let $\lambda_1\leq \dots \leq \lambda_n$  be the variational eigenvalues of $\L_p$, and let $\lambda$ be any eigenvalue such that $\lambda<\lambda_k$. Consider any eigenfunction $f$ corresponding to $\lambda$.
 	Let   $A_1,\dots,A_m$ be the strong nodal domains
 	of $f$ and let $F$ be the corresponding strong nodal space.
 	Lemma \ref{le:NDS} implies that 
  $\max_{g \in F}\R_{p}(g)\leq \lambda$. 
 	As the functions $f|_{A_1}, \dots, f|_{A_m}$ are linear independent we have $\gamma(F\cap \SS_p)=m$. 
 	In particular $F\cap \SS_p \in \mc F_m(\SS_p)$ and by the definition of $\lambda_m$ we get
 	\begin{equation}\label{eq:bounds} 
 	\lambda_m \leq \max_{g \in F\cap \SS_p}\R_p(g)
 	\leq \lambda< \lambda_{k}\, .
 	\end{equation}
 	As a consequence  we have $\lambda_m< \lambda_{k}$ which implies $m\leq k-1$. 
 \end{proof}

 	

 For the weak nodal domains we need a few additional remarks. Let $A_1, \dots, A_m$ be the weak nodal domains of $f$. Since $\cup_i A_i = V$ and $G$ is connected, then for any $i$ there exists $j$ such that $A_i\approx A_j$. Moreover, the following lemma holds
 \begin{lemma}
 	Let $A$ and $B$ be two weak nodal domains induced by the non-constant eigenfunction $f:V\to \RR$, such that $A\approx B$. Then there exist $u \in A$ and $v \in B\setminus A$ such that $u \sim v$.
 \end{lemma}
 \begin{proof}
 	If $A\cap B =\emptyset$ the statement is straightforward. Assume that $A \cap B\neq \emptyset$. By definition we have $f(u) = 0$, for any $u \in A \cap B$, thus for any such $u$ it holds $0 =\lambda\, \mu(u)\Phi_p(f(u))= \sum_{v \in V}w(uv)\Phi_p(f(u)-f(v))=\sum_{v \in V}w(uv)\Phi_p(f(v))$. Note that, by definition, as $u \in  A \cap B$, then any $v$ such that  $v\sim u$ is either in $ A$ or in $B$. As  $w(uv)>0$ when $u \sim v$, the values $\Phi_p(f(v))$ have to be either all zero or both positive and negative. However, the maximality of the nodal domains implies  that  $\Phi_p(f(v))$ can not be zero for all $v \sim u$ and all $u \in  A \cap B$. Then there exists $v \in  A\cup B$ such that $v \sim u$ and $f(v) \neq 0$. This concludes the proof.
 \end{proof}
 
 It is clear that  adjacent nodal domains have different sign. Then we deduce from the above lemma that, given any two  adjacent weak nodal domains $A\approx B$ of an eigenfunction $f$, two cases are possible:
 \begin{enumerate}
 	\item[P1.] There exist $u \in A$ and $v \in B$ such that $u \sim v$ and $f(u)f(v)<0$. 
 	\item[P2.] $f(u)f(v)=0$ for all $u \in A$ and $v \in B$ such that $u\sim v$, and there exist $u \in A$ and $v \in B$ such that $u \sim v$, $f(u)=0$ and $f(v)\neq 0$. 
 \end{enumerate}
 \vspace{3pt}
 
\begin{proof}[Proof of Theorem \ref{thm:p_laplacian-weak-domains}]	
 	Let $f$ be an eigenfunction of $\lambda_k$ and let $A_1, \dots, A_m$ be the weak nodal domains of $f$.  Suppose by contradiction that $m>k$.  On the other hand we deduce from Lemma \ref{le:NDS} that inequality \eqref{eq:bounds} holds also for the weak nodal domains. Namely, for any $g$ in the weak nodal space $F$ of $f$, we have  $   \max_{g \in F} \R_p(g) = \max_{g \in F\cap \SS_p} \R_p(g) \leq \lambda_k$. 
 	Observe that, as $m>k$ and $f \in F$, we have $F =\mathrm{span}\{f\}\oplus H$, for some $H$ such that  $\dim H \geq k$. In particular $m = \gamma(F\cap \SS_p)$ and $k \leq \gamma(H\cap \SS_p)$. As a consequence $H\cap \SS_p\in \mc F_k(\SS_p)$ and we get  
 	\begin{equation*}
 	\lambda_k \geq \max_{g \in F\cap \SS_p}\R_p(g)\geq \max_{g \in H\cap \SS_p}\R_p(g)\geq \min_{X \in \mc F_k(\SS_p)}\max_{g \in X}\, \R_p(g)=\lambda_k\, .
 	\end{equation*}
 	Thus the relations above hold with equality and we deduce that $H\cap \SS_p$ is a minimizing set, and by Lemma  \ref{thm:existence} there exists an eigenfunction  $g = \sum_{s=1}^m \alpha_s\, f|_{A_s} \in H$.

 	As $\R_p(g)$ is the maximum of the Rayleigh quotient on $H$ 
 	we deduce from the proof of Lemma \ref{le:NDS}  that $\sum_{i \neq j} \sum_{u \in A_i}\sum_{v \in A_j} w(uv) F_{ij}(u,v)=0$, 
	where
	
	$$F_{ij}(u,\!v)\!=\!
	\big|\alpha_i f|_{A_i}\!(u) - \alpha_j f|_{A_j}\!(v)\big|^p \!- \Big(\!|\alpha_i|^p\big|f|_{A_i}\!(u)\big|+ |\alpha_j|^p\big|f|_{A_j}\!(v)\big|\!\Big)\! \big|f|_{A_i}\!(u) - f|_{A_j}\!(v)\big|^{p-1}$$

	%
 	By Lemma \ref{lem:ax-by} each of the summands $w(uv)F_{ij}(u,v)$ is nonpositive, then all of them have to vanish individually. Choose any pair of adjacent sets $A_s \approx A_r$. If they satisfy property P1 above, then there exist $u \in A_s$ and $v\in A_r$ such that $w({uv})>0$ and $f|_{A_s}(u)f|_{A_r}(v)<0$. Therefore $w(uv)F_{sr}(u,v)=0$
 	implies  $\alpha_s = \alpha_r$, by virtue of Lemma \ref{lem:ax-by}. 
 	
 	If P1 does not hold, then P2 holds. 
 	Since $g$ is an eigenfunction of $\L_p$, for any $\beta \in \RR$, we have the following entrywise equations
 	\begin{align*}
	\lambda_k\mu(u)\Phi_p(\beta f(u)) &=\sum_{v \in V} w(uv)\Phi_p(\beta f(u)-\beta f(v)), \quad u \in V \\
	\lambda_k\mu(u)\Phi_p(g(u)) &= \sum_{v\in V} w(uv)\Phi_p(g(u)-  g(v)), \quad u\in V .
 	\end{align*}
 	 As P2 holds for $A_s$ and $A_r$, then there exist  $u \in A_s$ and $v \in A_r$  such that $u\sim v$, $f(u)=0$ and $f(v)\neq 0$. Then $\beta f(u) = g(u) =0$ and  the previous equations 
 	 imply
 	$$\sum_{v\in V} w({uv})\{\Phi_p(\beta f(v))-\Phi_p(g(v))\}=0\, .$$  
 	The quantities $w(uv)$ are zero unless $v\sim u$. Since $f(u)=0$,  the maximality of the nodal domains implies that all the vertices  $v$ adjacent to $u$ are either in $A_s$ or in $A_r$. We have
 \begin{multline*}
  \sum_{v \in A_s} w(uv)\{\Phi_p\big(\alpha_s f|_{A_s}(v)\big)-\Phi_p\big(\beta f|_{A_s}(v)\big)\}\\
  =\sum_{v \in A_r} w(uv)\{\Phi_p\big(\beta f|_{A_r}(v)\big)-\Phi_p\big(\alpha_r f|_{A_r}(v)\big)\}.
 \end{multline*}
 	Thus choosing $\beta = \alpha_s$ we get 
 	$\{\Phi_p(\alpha_s)-\Phi_p(\alpha_r)\}\sum_{v \in A_r} w(uv)\Phi_p(f|_{A_r}(v)) =0$.
 	Since $w(uv)\geq 0$ for all $v \in A_r$, there exists $x\in A_r$ such that   $w(ux)>0$,  $f(x)\neq 0$, and  the entries of $f|_{A_r}$ have same sign, then the previous identity implies $\Phi_p(\alpha_s)-\Phi_p(\alpha_r)=0$, that is $\alpha_s=\alpha_r$.
 	We finally conclude that, if $A_s\approx A_r$, then $\alpha_s=\alpha_r$.  The connectedness of the graph implies then $ \alpha = \alpha_1 = \cdots = \alpha_m$, and we obtain $g=\sum_{s=1}^m \alpha_s\,  f|_{A_s} =\alpha f$. This gives a contradiction as by  construction $g\in H$ is linear independent with respect to $f$. 
 \end{proof}

We show in the following that the bounds cannot be improved in general, by discussing the nodal domain structure of an example graph. 

\subsection{Nodal domains of the eigenfunctions of the path graph}\label{sec:Pn}

It is well known that for $p=2$, the upper bounds shown in the nodal theorem are tight, for any $k$. Simple examples where the  those bounds are achieved for $p=2$ are the line segment, in the continuous setting, and the   path graph 
	$$
	P_n\, =\, \begin{grafo}{4mm}
	(1,0)\NodoO[2mm]{0}{}{R}{},
	(3,0)\NodoO[2mm]{1}{}{R}{},
	(5,0)\NodoO[2mm]{2}{}{R}{},
	(7,0)\NodoO[2mm]{3}{}{R}{},
	(8,0)\Nodo[3mm]{x}{\cdots}{R}{},
	(9.2,0)\Nodo[3mm]{xx}{\cdots}{R}{},
	(10.4,0)\Nodo[2mm]{xxx}{\cdots}{R}{},
	(11.6,0)\Nodo[2mm]{xxxx}{\cdots}{R}{},
	(12.6,0)\NodoO[2mm]{n}{}{R}{},
	\arcoD[]{0}{1}
	\arcoD[]{1}{2}
	\arcoD[]{2}{3}
	\end{grafo}$$
	in the discrete case. For convenience, throughout this section we identify $V$ with the integers set $\{1,\dots,n\}$, and we fix both the vertex and the edge measures to be constantly one. 

 The eigenfunctions $f_k(x)$ of the continuous  $p$-Laplacian on the line segment are known to be given for $p>1$ by 
$f_k(x)=\sin_p(kx)$, where $\sin_p(x)$ is a special periodic function  \cite{Elbert1979,Otani1984}.
However, dissimilar to the case $p=2$,  a direct computation reveals that the functions obtained by evaluating $f_k(x)$ on a uniform grid, are not the eigenfunctions of $\L_p$ on $P_n$ 
for $p\neq 2$. The reason is that when $p\neq 2$, there is no addition formula 
relating $\sin_p$ and its derivative \cite{Lindqvist1995}. 
While an explicit formula for the eigenfunctions of $\L_p$ on $P_n$ when $p\neq 2$ is out of reach, we devote the remaining part of this section to show that the variational eigenpairs of the $p$-Laplacian on $P_n$ have several special properties, and in particular we prove that the number of nodal domains induced by the eigenfunction of the variational eigenvalue $\lambda_k$ on $P_n$, is exactly $k$.

For a function $f:V\to \RR$ let us define $\tilde f:\mathbbm Z\to\RR$ as follows: first we define $g$  by $g(i)=f(-i+1)$ for $i=0,-1,\dots, -n+1$ and $g=f$ over $V$; then we define $\tilde f$ by extending $g$ periodically over $\mathbbm Z$. The extension $\tilde f$ allows us to recast the eigenvalue equation \eqref{eq:eigenequation} as the
infinite system of nonlinear equations 
\begin{equation}\label{eq:difference-equation}
 \mathcal H(\lambda,\tilde f, k)=D \Phi_p D\tilde f (k)- \lambda\,  \Phi_p(\tilde f(k+1))=0, \quad k\in \mathbbm Z
\end{equation}
where $D$ is the forward difference operator defined by $D f(k) = f(k+1)-f(k)$. One easily verifies that 
 \begin{equation}\label{lapl}
 	(\L_p f)(u) =\lambda\, \Phi_p(f(u)), \quad \forall u \in V \;\Longleftrightarrow\; \mathcal H(\lambda,\tilde f, k)=0, \quad \forall k\in \mathbbm Z\, .
 \end{equation}

It turns out that \eqref{eq:difference-equation} is a particular version of a famous non-linear difference equation that has been studied quite intensively in the difference and differential equations literature (see e.g. \cite[Chap.\ 3]{HandbookODE2002}). 
 In the following any interval $[a,b]$ is meant to be discrete, i.e.\ $[a,b]=\{x \in \mathbbm Z : a\leq x \leq b\}$. We shall say that  $(a,a+1]$ is a generalized zero for $f$ if $f(a) \neq 0$ and $f(a) f(a+1)\leq 0$. Equation \eqref{lapl} is said to be disconjugate on $[a,b]$ provided that any solution of this equation has at most one generalized zero on $(a,b+1]$ and the solution $f$ satisfying $f(a)=0$ has no generalized zeros on $(a,b+1]$. The following generalized version of Sturm's comparison theorem is due to  Reh\'ak \cite[Thm.~2]{Rehak2001}.
\begin{theorem} \label{thm:rehak}Let $p>1$,  $\eta\geq \lambda$ and let $\tilde f,\tilde g$ be sequences such that $\mathcal H(\lambda,\tilde f, k)=\mathcal H(\eta,\tilde g, k)=0$ for $s \leq k \leq t$. If $\tilde g$ is disconjugate on $[s,t]$ then $\tilde f$ is disconjugate on $[s,t]$ as well. 
\end{theorem}
We have as a direct consequence
\begin{lemma}\label{lem:genzero}
	Let $(\lambda,f)$ and $(\eta, g)$ be two eigenpairs of $\L_p$ on $P_n$ with $\eta\geq \lambda$. If $f$ and $g$ have the same number of generalized zeros, then the generalized zeros of $f$ and $g$  coincide.
\end{lemma}
\begin{proof} The proof is a direct consequence of Theorem \ref{thm:rehak}. We briefly sketch the argument.
	Let  $(a_1, a_1+1]$, $\dots$, $(a_k, a_k+1]$ and $(b_1, b_1+1]$, $\dots$, $(b_k, b_k+1]$ be the generalized zeros of $g$ and $f$ respectively, ordering them so that $2 \leq a_i+1\leq a_{i+1}\leq n-1$ and $2\leq b_{i}+1\leq b_{i+1}\leq n-1$, for $i=1, \dots, k$. Using the symmetry of $\tilde f $ and $\tilde g$ one observes that $b_1\geq a_1$, as otherwise Theorem \ref{thm:rehak} would be contradicted. This implies that $b_2\geq a_2$, as $b_2<a_2$ would imply that $g$ is disconjugate on $[b_1, a_{2}+1]$, while $f$ is not. Proceeding by induction we have $b_i \geq a_i$ for $i=1, \dots, k$. A similar argument shows that $b_k\leq a_k$, thus $b_i\leq a_i$ for $i=1, \dots, k$.
\end{proof}
Let us make a few further remarks. Let $f$ be an eigenfunction on $P_n$. Then $f(1)f(n)\neq 0$. Indeed $f(1)=0$ implies $0=\mathcal H(\lambda,\tilde f, 0)=\Phi_p\big(f(2)\big)$ and thus $f(i)=0$ for all $i \in V$. Similarly for $f(n)$. Moreover, the next lemma shows that all variational eigenvalues of $P_n$ are distinct.
\begin{lemma}\label{thm:thight}Let $p>1$ and let $0=\lambda_1<\lambda_2\leq \dots \leq \lambda_n$ be the variational eigenvalues of $\L_p$ on $P_n$. Then $0<\lambda_2 <\lambda_3 <\dots <\lambda_n$. 
\end{lemma}
\begin{proof}
Suppose $\lambda$ and $\eta$ are two variational eigenvalues with $\lambda = \eta$ and let $f$ be any  eigenfunction of $\lambda$. Arguing as in the proof of Theorem \ref{thm:p_laplacian-weak-domains}, there exists an eigenfunction $g$ of $\eta$  which is linear independent with respect to $f$. We can assume w.l.o.g.\  that $f(1)=g(1)=1$. Then 
$$\mathcal H(\lambda,\tilde f, 0) = \Phi_p\big(1-f(2)\big)-\lambda  = \Phi_p\big(1-g(2)\big)-\lambda  =\mathcal H(\lambda,\tilde g, 0) $$
	implying $f(2)=g(2)$. By induction we get $f=g$, leading to a contradiction. 
\end{proof}

The following theorem, finally, gives a complete description of the $p$-Laplacian  nodal domains of the path graph, for any $p>1$.
\begin{theorem} 
Let $p>1$ and let $f$  be an eigenfunction corresponding to the variational eigenvalue $\lambda_k$ of $\L_p$ on $P_n$. Then the number of zero entries of $f$ is at most $k-1$, and it induces exactly $k$ weak and strong nodal domains.	 
\end{theorem}
\begin{proof}
	Let us write $\nu(f)$  to denote  the number of weak nodal domains of $f$. Observe that an eigenfunction of $\L_p$ on the path graph cannot vanish on two consecutive entries
	as otherwise it would be constant zero. Indeed, if $i$ is such that $f(i)\neq 0$ and $f(i+1)=0$, then by \eqref{lapl} we have $\Phi_p(f(i+2))=-\Phi_p(f(i))$.  This implies that the number of zero entries of $f$ is at most $\nu(f)-1$.   
	
		Let us show that  $\nu(f)=k$. The statement is true for $k=1,2$ due to Corollary \ref{cor:second-eigenvector}.   We proceed by induction. For $k>2$ assume that $\nu(f)=k-1$ for any  eigenfunction $f$ of $\lambda_{k-1}$, and let $g$ be an eigenfunction of $\lambda_{k}$. Note that, as the multiplicity of each $\lambda_k$ is one, by the nodal theorem if follows that $\nu(g)\leq k$. Arguing as in Lemma \ref{lem:genzero} using Theorem \ref{thm:rehak}, one observes that the  the overall number of generalized zeros of $g$ cannot be less than the one of $f$. If follows that $\nu(g)\geq k-1$. To complete the proof we show that $\nu(g)\neq k-1$.  To this end, we assume that $\nu(g)=k-1$ and we show that this  implies a contradiction.	
	  Equation \eqref{lapl} for $f$ and $g$ becomes
		\begin{align}\label{lapl2}
		\lambda_{k-1}\,  \Phi_p\big(f(i)\big)&=\Phi_p\big(f(i)-f(i+1)\big)-\Phi_p\big(f(i-1)-f(i)\big)\\
		\lambda_{k} \,   \Phi_p\big(g(i)\big)&=\Phi_p\big(g(i)-g(i+1)\big)-\Phi_p\big(g(i-1)-g(i)\big)\, \, .\label{lapl3}
		\end{align}
		Consider the set $V_+=\{i\in V : f(i)g(i)\neq 0\}$. We show by induction that the following inequalities hold
		\begin{equation}\label{eq:thesis}
		 \Phi_p\(1-\frac{f(i+1)}{f(i)}\)< \Phi_p\(1-\frac{g(i+1)}{g(i)}\), \quad \forall i\in V_+\setminus\{n\}\, .
		\end{equation}
		As $k-1= \nu(f)=\nu(g)$, then Lemma \ref{lem:genzero} implies that  $f$ and $g$  have the same generalized zeros.  Since $\tilde f(0)=f(1)$ and $\tilde g(0)=g(1)$, from $\lambda_{k-1}<\lambda_k$, \eqref{lapl2}  and \eqref{lapl3}  we get 
		$\Phi_p(1-f(2)/f(1))< \Phi_p(1-g(2)/g(1))$. 
		 We have $1 \in V_+$ and \eqref{eq:thesis} holds for $i=1$. We assume that $i-1\in V_+$ satisfies \eqref{eq:thesis}, and show that the same holds for the next index in $V_+$. 
		 There are two possible cases: either $i \in V_+$, which we discuss next, or $i \not \in V_+$, which we discuss below.
	
		If $i \in V_+$ then $f(i)g(i)\neq 0$ and we derive from $\lambda_{k-1}<\lambda_k$, \eqref{lapl2} and \eqref{lapl3}  that 
		\begin{equation}\label{x}
		\Phi_p\!\(\!\!1-\frac{f(i+1)}{f(i)}\!\!\) \!- \Phi_p\!\(\!\!1-\frac{g(i+1)}{g(i)}\!\!\)\!<\!\Phi_p\!\!\(\!\frac{f(i-1)}{f(i)}-1\!\!\right)  \!- \!\Phi_p\!\left(\!\!\frac{g(i-1)}{g(i)}-1\!\!\right)
		\end{equation} 
		Note that, as $f$ and $g$ have the same generalized zeros and  $f(i)g(i)f(i-1)g(i-1)\neq 0$, then  $f$ and $g$ have the same sign on $[i-1,i]$. Therefore $i-1\in V_+$ and 
		  \eqref{x} imply  
		that  \eqref{eq:thesis} holds for $i\in V_+$.
		
		Now let us discus the case $i\notin V_+$. Note that, as \eqref{eq:thesis} holds for $i-1\in V_+$, $f(i)$ and $g(i)$ can not be both zero.	
		 
		The case $g(i)=0$ and $f(i)\neq 0$ is not possible. In fact, as \eqref{eq:thesis} holds for $i-1\in V_+$, then $\Phi_p\(1-f(i)/f(i-1)\right)< 1$, 
		showing that $f(i)f(i-1)>0$. On the other hand $g(i)=0$ implies  
		that $g$ has a generalized zero on $(i-1, i]$, yielding a contradiction.
		
		Finally, if $f(i)=0$ and $g(i)\neq 0$, then as \eqref{eq:thesis} holds for $i-1$
		we have 
		$g(i-1)g(i)< 0$. Therefore $(i-1,i]$ is a generalized zero for $g$. Now note that $f(i)=0$ implies that $(i, i+1]$ is not a generalized zero of $f$. Thus,  by Theorem \ref{thm:rehak},  $g(i)g(i+1)>0$.  
	  We deduce that $i+1 \in V_+$ and  from $\lambda_{k-1}<\lambda_k$, \eqref{lapl2} and \eqref{lapl3}, we get 
		\begin{equation}\label{xxx}
		\Phi_p\(1-\frac{f(i+2)}{f(i+1)}\right)-\Phi_p\(1-\frac{g(i+2)}{g(i+1)}\right)<-1-\Phi_p\(\frac{g(i)}{g(i+1)}-1\right)
		\end{equation}
		As $g(i)g(i+1)>0$ we have 
		$ \Phi_p\({g(i)}/{g(i+1)}-1\right)>-1$
		and we obtain from \eqref{xxx} that \eqref{eq:thesis} holds for $i+1\in V_+$.
		
		Now observe that  we can proceed 
		the other way round to show that the following sequence of inequalities holds as well
	   \begin{equation}\label{yy}
		\Phi_p\left(1-\frac{f(i-1)}{f(i)}\right)<  \Phi_p\left(1-\frac{g(i-1)}{g(i)}\right), \quad \forall i \in V_+\setminus\{1\}\, .
		\end{equation} 
		In fact, since $f(n)=\tilde f(n+1)\neq 0$ and $g(n)=\tilde g(n+1)\neq 0$, then $n\in V_+$ and \eqref{yy} holds for $i=n$. Thus we have the basis for the induction and we can repeat the same argument as before.  
		To conclude we observe that there exist two consecutive 
		indices $m$ and $m+1$ in  $V_+$, thus by plugging  $i = m$ into \eqref{eq:thesis} and $i=m+1$ into \eqref{yy} we obtain a contradiction. To this end note that, as $f(1)\neq 0$,  if there are no consecutive indices in $V_+$, then  $f(i)=0$ for all even indices $i$. Therefore \eqref{lapl2} implies that $f(i)$ is nonzero for $i$ odd, and we get $\lambda_{k-1}\Phi_p(f(1))=\Phi_p(f(1))$ and $\lambda_{k-1}\Phi_p(f(3))=2\Phi_p(f(3))$, which is not possible.
\end{proof}

\section{Nodal properties of the $1$-Laplacian}\label{sec:L1} 

We devote this section to discuss the non-smooth case of the $1$-Laplacian which becomes a set-valued operator.
With the set-valued sign operator $\Sign(x)=\{1\}$ if $x > 0$, $\Sign(x)=\{-1\}$, if $x <0$ and $\Sign(x)=[-1,1]$ for $x=0$,
it is then straightforward to verify that the $1$-Laplacian is the operator realizing the following entrywise identity \cite{Hein2010},
$$(\L_1 f)(u) = \Big\{\sum_{v\in V} w(uv)z(uv) \mid z(uv)=-z(vu), \, z(uv) \in \Sign(f(u)-f(v))\Big\}\, .$$ 
The corresponding eigenequation \cite{Hein2010,Chang2014} reads
\[ 0 \in (\L_1 f)(u) - \lambda \,\mu(u) \Sign(f(u)).\]
It has been shown that this is a necessary condition for a critical point of the associated non-smooth Rayleigh quotient $\R_1$ \cite{Hein2010} via the Clarke subdifferential, and more recently also to be sufficient \cite{Chang2014}.

 The classical  Lusternik-Schnirelman theory 
 can be extended to the case of a locally Lipschitz functional (see \cite[Sec.\ 3]{Chang1981} and \cite{Chang2014}) and provides a  variational characterization of the 
 spectrum of $\L_1$.  In particular, as for $p>1$, the following sequence
$$\lambda_k^{(1)} = \min_{A \in \mathcal F_k(\SS_1)}\max_{f \in A}\, \R_1(f), \qquad k = 1,\dots, n$$
defines a set of $n$ variational eigenvalues of $\L_1$. 

Our nodal domain theorems carry over the case $p=1$, but in a weaker form. The main difference is that the number of weak nodal domains of the $k$-th variational eigenfunction is upper bounded by $k+r-1$ (where $r$ is the multiplicity of the corresponding eigenvalue $\lambda_k^{(1)}$) instead of $k$, as for $p>1$. 
Note indeed that the proof of the Theorem \ref{thm:p_laplacian-strong-domains} holds unchanged if $p=1$. This is not the case for the weak nodal domains. 
Thus, the nodal domain theorem for the $1$-Laplacian reads as follows
\begin{theorem}\label{thm:1-L}
Let $G$ be connected and $0=\lambda_1< \dots \leq \lambda_n$ 
be the variational eigenvalues of the $1$-Laplacian.  If $\lambda_k$ has multiplicity $r$, then any eigenfunction  of $\lambda_k$ induces at most $k+r-1$ strong and weak nodal domains. 
\end{theorem}

We finally show that Theorem \ref{thm:1-L} is tight by discussing the eigenfunctions for $p=1$ of the unweighted path graph.
%
For the sake of simplicity we consider this time the path graph  $P_3$ on three vertices 
\[ P_3=\begin{grafo}{4mm}
(0,0)\NodoO[4mm]{0}{v_1}{R}{},
(6,0)\NodoO[4mm]{1}{v_3}{R}{},
(3,0)\NodoO[4mm]{2}{v_2}{R}{},
\arcoD[]{0}{2}
\arcoD[]{1}{2}
\end{grafo} .
\]
With $z(v_2v_1) \in \Sign(f(v_2)-f(v_1))$ and $z(v_2v_3)\in\Sign(f(v_2)-f(v_3))$ we get the following system of equations for the 
eigenvalues and eigenfunctions of $\L_1$,  when $\mu(u)=\sum_{v}w(uv)$
$$
\left\{
\begin{array}{l}
-z(v_2v_1) \in \lambda\,  \Sign(f(v_1))\\
z(v_2v_1)+z(v_2v_3)  \in  2\, \lambda\,  \Sign(f(v_2))\\
-z(v_2v_3) \in \lambda\,  \Sign(f(v_3))
\end{array}
\right.
$$
We show in the following that any non-constant eigenfunction has eigenvalue $\lambda=1$. To this end we make a case distinction.
If $f(v_1)>0$, we have the cases
\begin{itemize}
	\item $f(v_2)<f(v_1)$ implies  $z(v_2v_1)=-1$, thus $\lambda=1$
	\item $f(v_2)>f(v_1)$ implies $z(v_2v_1)=1$, thus $\lambda=-1$, which is a contradiction as $\lambda \geq 0$
	\item $f(v_2)=f(v_1)>0$ implies $z(v_2v_1)=-\lambda$ and thus $z(v_2 v_3)=3\lambda$ together with $-z(v_2 v_3) \in \lambda \Sign(f(v_3))$ leads to a contradiction for $\lambda >0$
\end{itemize}
Similarly, if $f(v_1)=0$ we have the cases
\begin{itemize}
	\item $f(v_2)>f(v_1)=0$ implies $1+z(v_2 v_3) = 2\lambda$. If $f(v_3)\leq 0$ one has $z(v_2 v_3)=1$ and this yields
        $\lambda=1$. If $f(v_3)>0$, then $z(v_2 v_3)=-\lambda$ and thus $1=z(v_2 v_1)=3\lambda$ which together with $z(v_2 v_1) \in (-\lambda,\lambda)$ leads to a contradiction.
        \item $f(v_2)=f(v_1)=0$ and $f(v_3)>0$ yields $z(v_2 v_3)=-1$ and thus $\lambda=1$.
\end{itemize}
These are, up to sign, all the cases ones has to consider. In all the cases one gets the eigenvalue $\lambda=1$. Thus the variational eigenvalues have to be $\lambda_2^{(1)}=\lambda_3^{(1)}=1$. One eigenfunction $f$ for the eigenvalue $\lambda=1$ is  given by $f(v_1)=-f(v_2)=f(v_3)=1$. This eigenfunction has three weak and strong nodal domains and thus the result for $p>1$ that the number of weak nodal domains of the $k$-th eigenvalue with multiplicity $r$ is upper bounded by $k$ does not hold for the case $p=1$. Moreover, our bound of $k+r-1=2+2-1=3$ is tight for the given example.

\section{A higher order Cheeger inequality via nodal domains}\label{sec:cheeger}	
	A set of multi-way Cheeger constants $h_k(G)$,
	$k = 2,3,\dots$ alternatively called high-order isoperimetric constants, has been recently studied by \cite{Daneshgar2010,Lee2012,Miclo2008}. 
	For $A \subseteq V$, let $E(A,\bar A)$ be the set of edges having one endpoint in $A$ and one in the complement of $A$, denoted as $\bar A$. 
	Consider the quantity
	$$c(A)=\frac{w(E(A,\bar A))}{\mu(A)}$$
	where the measure of a discrete set is given by the sum of the weights of the elements in the set. Finally let $\mc D_k(G)$ be the set of $k$ non-empty, mutually disjoint subsets of $V$, $\mc D_k(G)=\{\emptyset \neq A_1, \dots, A_k \subseteq V: A_i\cap A_j =\emptyset\}$.  The higher-order isoperimetric constants $h_k(G)$ are defined as
	$$h_k(G) = \min_{\mc A\in \mc D_k(G)}\max_{A \in \mc A}\, c(A)$$

We conclude the paper by exploiting the relation among  the high-order ispoerimetric constants, the variational  eigenvalues of $\L_p$ and their nodal domains.
\begin{theorem}\label{thm:cheeger-domains}
For $p>1$, let $f:V \to \RR$ be an eigenfunction of $\L_p$ corresponding to the variational eigenvalue $\lambda_k^{(p)}$, and let $m$ be the number of its strong  nodal domains. Then 
 $$\frac{2^{p-1}}{\tau(G)^{p-1}}\, \frac{h_m(G)^p}{p^p} \,\leq\, \lambda_k^{(p)}\,\leq\, 2^{p-1}\,h_k(G)$$
 where $\tau(G)=\max_{u \in V} \frac{d(u)}{\mu(u)}$ and $d(u)=\sum_{v \in V} w(uv)$ is the degree of the vertex $u$.
\end{theorem}
This theorem is a {direct} generalization of Theorem 5 in Daneshgar et al \cite{Daneshgar2010}, where the result was proven for the linear graph Laplacian ($p=2$) and $\mu(u)=d(u)$.
For $k=2$ the result has been shown in \cite{Amghibech2003,Thomas2009} for $p>1$ and it has been noted there that the inequality becomes tight
as $p\rightarrow 1$ as the second eigenfunction always has two strong nodal domains given that the graph is connected. The equality for $p=1$ and $k=2$ has been shown in \cite{Hein2010}, see also \cite{Chang2014}. For $k>2$ the situation changes as now an extra condition is required in order that the higher order Cheeger inequality becomes tight for $p\rightarrow 1$. Namely, as $p$ approaches one, the number of {strong} nodal domains of the eigenfunction corresponding to
the variational eigenvalue $\lambda_k$ has to become equal to $k$. As discussed in the preceding section, the unweighted path graph is a graph with this property. However it is known that, when $G$ is not a tree, the number of nodal domains of the eigenfunctions of $\lambda_k^{(2)}$ is in general less than $k$. In fact, for $p=2$, the number of strong nodal domains induced by any eigenfunction $f$ of $\lambda_k^{(2)}$ is at least $k+r-1-\ell-z$, where $z$ is the number of vertices where $f$ is zero, and $\ell$ the minimal number of edges that need to be removed from $G$ in order to turn it into a tree \cite{Berkolaiko2008,Xu2012}. It remains an interesting open problem to generalize these lower bounds to the nonlinear case $p\neq 2$.

The proof of Theorem \ref{thm:cheeger-domains} relies on the following Lemma which is of independent interest. 
\begin{lemma}\label{lem:key1}
 For any  $f:V\to \RR$ and any $p>1$, there exists  $A\subseteq \{u : f(u)\neq 0\}$ such that 
 $$\R_p(f)\geq \left(\frac{2}{\tau(G)}\right)^{p-1}\left(\frac{c(A)}{p}\right)^p$$
\end{lemma}
\begin{proof} 
Consider the sets $E_0 = \{uv \in E: |f(u)|^p-|f(v)|^p=0\}$,  $E_+ = \{uv \in E: |f(u)|^p-|f(v)|^p>0\}$ and, for  $\lambda\geq 0$,   $A_\lambda = \{u \in V : |f(u)|^p>\lambda\}$. By changing the order of summation and integration, and by the definition of $\mu(A_\lambda)$ we have
\begin{equation}\label{eq:0}
 \int_0^\infty \mu(A_\lambda)d\lambda = \int_0^\infty \sum_{u\in A_\lambda}\mu(u)d\lambda=
	\sum_{u \in V}\int_0^{|f(u)|^p}\mu(u)d\lambda=\|f\|_{\ell^p(\V)}^p
\end{equation}
	Now we derive an upper bound for $\int_0^{\infty}w(E(A_\lambda,\bar {A_\lambda}))d\lambda$. 
	Exchanging the role of integration and summation, as before, we have
	\begin{equation}\label{eq:a}
		\int_0^\infty w(E(A_\lambda, \bar{A_\lambda}))d\lambda = 
		\sum_{uv \in E_+}w(uv)\int_{|f(v)|^p}^{|f(u)|^p}d\lambda
		= \frac 12 \sum_{uv \in E}w(uv)\Bigl| |f(u)|^p-|f(v)|^p\Bigr|\, .
	\end{equation}
   Moreover, if $q$ is the H\"older conjugate of $p$, then  H\"older's inequality implies  
\begin{gather}
	\frac 12 \sum_{uv \in E}w(uv)\Bigl| |f(u)|^p-|f(v)|^p\Bigr|
	=\sum_{uv \not \in E_0}\frac{w(uv)}{2} \, | f(u)- f(v)|  \left|\frac{|f(u)|^p - |f(v)|^p}{f(u) -f(v)}\right| \notag \\
	\leq  \, \left\{ \frac 1 2 \sum_{uv\in E}w(uv)|f(u)-f(v)|^p\right\}^{\frac 1 p}\left\{ \sum_{uv \not \in E_0} \frac{ w(uv)}{2}\left|   \frac{|f(u)|^p - |f(v)|^p}{f(u) -f(v)}\right| ^q \right\}^{\frac 1 q}.\label{eq:FIN}
\end{gather}
We use now the following inequality,  holding  for any $x, y \in \RR$ and $p>1$  \cite{Amghibech2003}, 
	$$ \left(  \frac1 p \left|\frac{|x|^p - |y|^p}{x-y}\right| \right)^q \leq \left(  \frac1 p \left|\frac{|x|^p - |y|^p}{|x|-|y|}\right| \right)^q \leq \frac{ |x|^p+|y|^p}{2} $$
	to get
\begin{align*}
		 \sum_{uv \not \in E_0} \frac{ w(uv)}{2}\left|   \frac{|f(u)|^p - |f(v)|^p}{f(u) -f(v)}\right| ^q 	
	&\leq \frac{p^q}{4} \sum_{uv \in E} w(uv) \bigl(|f(u)|^p+|f(v)|^p\bigr)\\ &\leq \frac{ p^q \,\tau(G)\,\| f\|_{\ell^p(\V)}^p}{2}\, .
	\end{align*}
	Thus,  together with \eqref{eq:0} \eqref{eq:a} and \eqref{eq:FIN}, we finally get the inequality
	$$\frac{\int_0^{\infty} w(E(A_\lambda, \bar{A_\lambda}))d\lambda}{\int_0^{\infty} \mu(A_\lambda)d\lambda}\, \leq\, p\, \left(\frac{\frac{1}{2}\sum_{uv}w(uv)|f(u)-f(v)|^p}{\sum_{u}\mu(u)|f(u)|^p}\right)^{1/p}\(\frac{\tau(G)}{2}\)^{1/q}\, .$$
	Since $w$ ad $\mu$ are positive functions, we get
	$$\frac{\int_0^{\infty} w(E(A_\lambda, \bar{A_\lambda}))d\lambda}{\int_0^{\infty} \mu(A_\lambda)d\lambda}\geq \inf_{\lambda \geq 0}\frac{w(E(A_\lambda, \bar{A_\lambda}))}{ \mu(A_\lambda)}$$ 
	which shows in turn that there exists  $\lambda_* \in [0,\infty)$ such that 
	$$c(A_{\lambda_*})\leq p\, \R_p(f)^{1/p}\(\frac{\tau(G)}{2}\)^{1/q}\, .$$ 
	Finally, as $A_{\lambda_*}\subseteq \{u:f(u)\neq 0\}$ by construction, the statement follows.
\end{proof}
\begin{proof}[Proof of Theorem \ref{thm:cheeger-domains}]
	Let $A_1, \dots, A_m$ be  the strong nodal domains of $f$.  Lemma \ref{le:NDS} implies $\lambda_k^{(p)}\geq \R_p(f|_{A_i})$,  for any $i=1, \dots, m$. 
		Moreover, by applying Lemma \ref{lem:key1},  we deduce that for any $i$ there exists {$B_i \subseteq A_i$} such that 
	$$\R_p(f|_{A_i})\,\geq\,  (2/\tau(G))^{p-1}\(c(B_i)/p\)^p\, .$$
	{As the nodal domains are disjoint and non-empty, they belong to $\mc D_m(G)$}.  We get 
\begin{equation*}
	\max_{i = 1, \dots, m}\R_p(f|_{A_i}) \geq \min_{\{B_i\} \in \mc D_m(G)}\max_{i=1, \dots, m}\,   \!\(\!\frac 2 {\tau(G)}\!\)^{\!p-1} \!\!\! \left(\!\frac{c(B_i)}{p}\!\right)^{\!p} \!\!=\(\!\frac 2 {\tau(G)}\!\)^{\!p-1}\!\!\!\left(\!\frac{h_m(G)}{p}\!\right)^{\!p}
\end{equation*}
which finishes the proof of the first inequality in the statement. For the second one, let $\chi_A$ denote the indicator function of $A\subseteq V$. Note that $\R_p(\chi_A)=c(A)$, and let $\{A_1^*,\ldots,A_k^*\}\subseteq \mc D_k(G)$ be such that $h_k(G)=\max_{i=1,\ldots,k}c(A_i^*)$. Let $\mc X$ be  the span of $\chi_{A_1^*}, \dots, \chi_{A_k^*}$. For any $g \in \mc X$, that is $g(u)=\sum_{i=1}^k \alpha_k \chi_{A_i}(u)$, 
we have 
\begin{align*}
\sum_{u\in V} \mu(u)|g(u)|^p
	=  \sum_{i=1}^{k} \sum_{u \in A^*_i}\mu(u)|\alpha_i \chi_{A^*_i}(u)|^p= \sum_{i=1}^{k} |\alpha_i|^p \sum_{u \in V} \mu(u)|\chi_{A^*_i}(u)|^p
	\end{align*}
	Using 
	the fact that $A_i^* \cap A_j^* =\emptyset$ for $i\neq j$, we get
	$$|g(u)-g(v)|^p = \Bigl|\sum_{i=1}^{k} \alpha_i (\chi_{A^*_i}(u)-\chi_{A^*_i}(v))\Bigr|^p \leq 2^{p-1}\sum_{i=1}^{k} |\alpha_i|^p |\chi_{A^*_i}(u)-\chi_{A^*_i}(v)|^p$$
	and we obtain as a consequence
	\begin{align*}
	\R_p(g)
	\leq 2^{p-1} \frac{\sum_{i=1}^{k} |\alpha_i|^p\sum_{uv} w(uv) |\chi_{A^*_i}(u)-\chi_{A^*_i}(v)|^p}{\sum_{i=1}^{k} |\alpha_i|^p \sum_{u} \mu(u)|\chi_{A^*_i}(u)|^p}\leq 2^{p-1}\max_{i=1,\dots, k}\R_p(\chi_{A^*_i})
	\end{align*}
	where we have used the inequality $(\sum_i a_i)/(\sum_i b_i)\leq \max_i a_i/b_i$, holding for  $a_i, b_i\geq 0$. Finally note that $\gamma(\mc X\cap \SS_p)=k$ by construction, therefore $\mc X\cap \SS_p \in \mc F_k(\SS_p)$ and  the latter inequality implies  $\lambda_k^{(p)}\leq \max_{g \in \mc X\cap\SS_p}\R_p(g)\leq 2^{p-1}h_k(G)$.
\end{proof}

\section{Acknowledgements}
This work has been supported by the ERC grant NOLEPRO.

\bibliography{./library}
\bibliographystyle{plain}
\end{document}